\documentclass[10pt]{article}

\makeatletter

\@addtoreset{equation}{section}
\makeatother

\usepackage[cp1252]{inputenc}
\usepackage[OT1]{fontenc}
\usepackage[english]{babel}

\usepackage{amssymb, amsmath, 
amsthm,
amscd,
graphicx,
blackdvi,
comment,
cite
}

\input xy
\xyoption{all}

\graphicspath{{figures/}}

\def\ds{\displaystyle}

\def\R{{\mathbb R}}
\def\N{{\mathbb N}}
\def\Z{{\mathbb Z}}

\def\tcut{t_{\operatorname{cut}}}
\def\tmax{t_1^{\operatorname{MAX}}}
\def\tt{\mathbf t}
\def\then{\quad\Rightarrow\quad}
\def\iff{\quad\Leftrightarrow\quad}

\newcommand{\E}{\operatorname{E}\nolimits}
\newcommand{\e}{\operatorname{e}\nolimits}
\newcommand{\SE}{\operatorname{SE}\nolimits}
\newcommand{\const}{\operatorname{const}\nolimits}
\newcommand{\Exp}{\operatorname{Exp}\nolimits}
\newcommand{\sgn}{\operatorname{sgn}\nolimits}
\newcommand{\am}{\operatorname{am}\nolimits}
\newcommand{\sn}{\operatorname{sn}\nolimits}
\newcommand{\cn}{\operatorname{cn}\nolimits}
\newcommand{\dn}{\operatorname{dn}\nolimits}
\newcommand{\Max}{\operatorname{Max}\nolimits}
\newcommand{\MAX}{\operatorname{MAX}\nolimits}
\newcommand{\CMAX}{\operatorname{CMAX}\nolimits}
\newcommand{\cl}{\operatorname{cl}\nolimits}
\newcommand{\SO}{\operatorname{SO}\nolimits}
\newcommand{\SU}{\operatorname{SU}\nolimits}
\newcommand{\SL}{\operatorname{SL}\nolimits}
\newcommand{\Id}{\operatorname{Id}\nolimits}

\def\a{\alpha}

\def\g{\gamma}
\def\lam{\lambda}
\def\f{\varphi}
\def\p{\psi}
\def\eps{\varepsilon}
\def\Del{\Delta}

\def\bnu{\bar{\nu}}

\def\tq{\widetilde{q}}

\def\lan{\langle}
\def\ran{\rangle}

\def\vh{\vec h}
\def\vH{\vec H}

\def\Cinft{C^{\infty}(T^* M)}

\def\ts{\,{\sn \tau}\,}
\def\tc{\,{\cn \tau}\,}
\def\td{\,{\dn \tau}\,}
\def\ss{\,{\sn  p}\,}
\def\cc{\,{\cn  p}\,}
\def\dd{\,{\dn  p}\,}

\def\tdp{\,{\dn^2 \tau}\,}
\def\ssp{\,{\sn^2  p}\,}
\def\ccp{\,{\cn^2  p}\,}
\def\ddp{\,{\dn^2  p}\,}

\def\Eo{\,{\E(p)}\,}

\newcommand{\pder}[2]{\frac{\partial \, #1}{\partial \, #2} }

\newcommand{\be}[1]{\begin{equation}\label{#1}}
\newcommand{\ee}{\end{equation}}
\newcommand{\eq}[1]{$(\protect\ref{#1})$}
\newcommand{\map}[3]{#1 \, : \, #2 \to #3}
\newcommand{\mapto}[3]{#1 \, : \, #2 \mapsto #3}
\renewcommand{\Vec}{\operatorname{Vec}\nolimits}
\newcommand{\ddef}[1]{#1}

\newcommand{\spann}{\operatorname{span}\nolimits}

\newtheorem{theorem}{Theorem}[section]
\newtheorem{lemma}{Lemma}[section]
\newtheorem{corollary}{Corollary}[section]
\newtheorem{proposition}{Proposition}[section]

\theoremstyle{remark}

\newcommand{\onefiglabelsize}[4]
{
\begin{figure}[htbp]
\begin{center}
\includegraphics[width=0.9\textwidth]{#1}
\\
\parbox[t]{0.8#4\textwidth}{\caption{#2}\label{#3}}
\end{center}
\end{figure}
}

\newcommand{\onefiglabel}[3]
{
\begin{figure}[htbp]
\begin{center}
\includegraphics[width=0.5\textwidth]{#1}
\\
\parbox[t]{0.5\textwidth}{\caption{#2}\label{#3}}
\end{center}
\end{figure}
}

\newcommand{\twofiglabel}[6]
{
\begin{figure}[htbp]
\includegraphics[width=0.47\textwidth]{#1}
\hfill
\includegraphics[width=0.47\textwidth]{#4}
\\
\parbox[t]{0.45\textwidth}{\caption{#2}\label{#3}}
\hfill
\parbox[t]{0.45\textwidth}{\caption{#5}\label{#6}}
\end{figure}
}

\title{Maxwell strata in sub-Riemannian problem on the group of motions of a plane
}

\author{
I. Moiseev\footnote{Via G. Giusti 1,  Trieste 34100, Italy, E-mail: moiseev.igor@gmail.com} 
\and 
Yu. L. Sachkov\footnote{Program Systems Institute,
Pereslavl-Zalessky,  Russia,
E-mail: sachkov@sys.botik.ru} 
}

\date{\today}

\begin{document}

\maketitle

\begin{abstract}
The left-invariant sub-Riemannian problem on the group of motions of a plane is considered. Sub-Riemannian geodesics are parametrized by Jacobi's functions. Discrete symmetries of the problem generated by reflections of pendulum are described. The corresponding Maxwell points are characterized, on this basis an upper bound on the cut time is obtained.

Keywords: optimal control, sub-Riemannian geometry, differential-geometric methods, left-invariant problem,  Lie group, Pontryagin Maximum Principle, symmetries, exponential mapping, Maxwell stratum

MSC: 49J15, 93B29, 93C10, 53C17, 22E30
\end{abstract}

\newpage
\tableofcontents

\newpage
\section{Introduction}
Problems of sub-Riemannian geometry have been actively studied by geometric control methods. 
One of the central and hard questions in this domain is a description of cut and conjugate loci. Detailed results on the local structure of conjugate and cut loci were obtained in the 3-dimensional contact case~\cite{agrachev_contact, gauthier_contact}. Global results are restricted to symmetric low-dimensional cases, primarily for left-invariant problems on Lie groups (the Heisenberg group~\cite{brock, versh_gersh},
 the growth vector $(n,n(n+1)/2)$~\cite{myasnich36, myasnichn(n+1)/2, monroyn(n+1)/2}, the groups $\SO(3)$, $\SU(2)$, $\SL(2)$ and the Lens Spaces~\cite{boscain_SO3}).
 
 The paper continues this direction of research: we start to study the left-invariant sub-Riemannian problem on the group of motions of a plane $\SE(2)$. This problem has important applications in robotics~\cite{laumond} and vision~\cite{petitot}. On the other hand, this is the simplest sub-Riemannian problem where the conjugate and cut loci differ one from another in the neighborhood of the initial point.

The main result of the work is an upper bound on the cut time $\tcut$ given in Theorem~\ref{th:tcut_bound_fin}: we show that for any sub-Riemannian geodesic on $\SE(2)$ there holds the estimate $\tcut \leq \tt$, where $\tt$ is a certain function defined on the cotangent space at the identity. In a forthcoming paper~\cite{cut_sre} we prove that in fact   $\tcut = \tt$.
The bound on the cut time is obtained via the study of discrete symmetries of the problem and the corresponding Maxwell points --- points where two distinct sub-Riemannian geodesics of the same length intersect one another.

This work has the following structure.
In Section~\ref{sec:PS} we state the problem and discuss existence of solutions.
In Section~\ref{sec:PMP} we apply Pontryagin Maximum Principle to the problem. The Hamiltonian system for normal extremals is triangular, and the vertical subsystem is the equation of mathematical pendulum.
In Section~\ref{sec:exp} we endow the cotangent space at the identity with special elliptic coordinates induced by the flow of the pendulum, and integrate the normal Hamiltonian system in these coordinates. Sub-Riemannian geodesics are parametrized by Jacobi's functions.
In Section~\ref{sec:max} we construct a discrete group of symmetries of the exponential mapping by continuation of reflections in the phase cylinder of the pendulum.
In the main Section~\ref{sec:estim}
we obtain an explicit description of Maxwell strata corresponding to the group of discrete symmetries, and prove the upper bound on cut time.
This approach was already successfully applied to the analysis of several invariant optimal control problems on Lie groups~\cite{dido_exp, max1, max2, max3, el_max, el_conj}.

\section{Problem statement}
\label{sec:PS}
The group of orientation-preserving motions of a two-dimensional plane is represented as follows:
$$
\SE(2) = 
\left\{\left(
\begin{array}{ccc}
\cos \theta & - \sin \theta & x \\
\sin \theta & \cos \theta & y \\
0 & 0 & 1
\end{array} 
\right)
\mid
\theta \in S^1 = \R / (2 \pi \Z), \ x, y \in \R
\right\}.
$$
The Lie algebra of this Lie group is
$$
\e(2) = \spann(E_{21}-E_{12}, E_{13}, E_{23}),
$$
where $E_{ij}$ is the $3 \times 3$ matrix with the only identity entry in the $i$-th row and $j$-th column, and all other zero entries.

Consider a rank 2 nonintegrable left-invariant  sub-Riemannian structure on $\SE(2)$,
i.e., a rank 2 nonintegrable left-invariant distribution $\Del$ on $\SE(2)$ with a left-invariant inner product $\langle \cdot, \cdot \rangle$ on $\Del$.
One can easily show that such a structure is unique, up to a scalar factor in the inner product. We choose the following model for such a sub-Riemannian structure:
\begin{align*}
&\Del_q = \spann(\xi_1(q), \xi_2(q)), \qquad \langle \xi_i, \xi_j\rangle = \delta_{ij}, \quad i, \ j = 1, 2,\\
& \xi_1(q) = q E_{13}, \quad  \xi_2(q) = q(E_{21}-E_{12}),
\end{align*}
and study the corresponding optimal control problem:
\begin{align}
&\dot q = u_1 \xi_1(q) + u_2 \xi_2(q), \qquad q \in M = \SE(2), \quad u = (u_1, u_2) \in \R^2, \label{sys2}\\
&q(0) = q_0 = \Id , \qquad q(t_1) = q_1,  \nonumber \\
&\ell=\int_0^{t_1}\!\!\!{\sqrt{u_1^2+u_2^2}}\;dt\rightarrow\min. \nonumber
\end{align}
In the coordinates $(x,y,\theta)$,  the basis vector fields read as
\be{xi12}
\xi_1 = \cos \theta \pder{}{x} + \sin \theta \pder{}{y}, \qquad \xi_2 = \pder{}{\theta},
\ee
and the
problem takes the following form:
\begin{align}
&\dot x = u_1 \cos{\theta}, \quad  \dot y = u_1 \sin{\theta}, \quad  \dot \theta = u_2, \label{sys1} \\
&q = (x,y,\theta) \in M \cong \R^2_{x,y} \times S^1_{\theta}, 
\quad u = (u_1, u_2) \in \R^2,
\label{qu} \\
&q(0) = q_0 = (0,0,0), \qquad q(t_1) = q_1 = (x_1,y_1,\theta_1),  \label{bound} \\
&\ell=\int_0^{t_1}\!\!\!{\sqrt{u_1^2+u_2^2}}\;dt\rightarrow\min. \label{l}
\end{align}

Admissible controls $u(\cdot)$ are measurable bounded, and admissible trajectories $q(\cdot)$ are  Lipschitzian.

The problem can be reformulated in robotics terms as follows. 
Consider a mobile robot in the plane that can move forward and backward, and rotate around itself (Reeds-Shepp car). The state of the robot is described by coordinates $(x,y)$ of its center of mass and by angle of orientation $\theta$. Given an initial and a terminal state of the car, one should find the shortest path from the initial state to the terminal one, when the length of the path is measured in the space $(x,y,\theta)$, see Fig.~\ref{fig:prob_state}.

\onefiglabel
{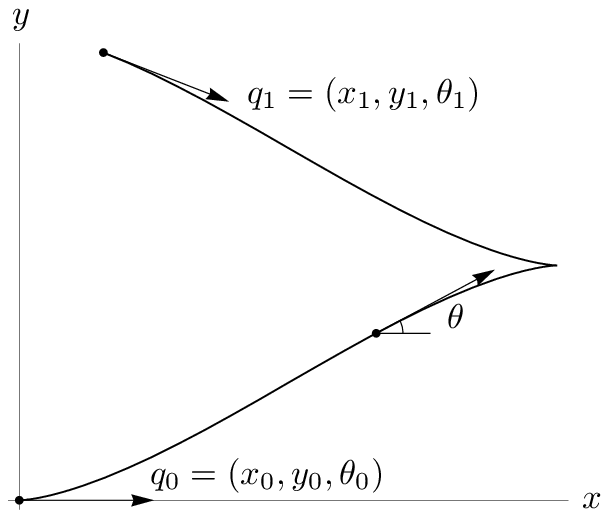}{Problem statement}{fig:prob_state}

Cauchy-Schwarz inequality implies that the minimization problem for the sub-Riemannian length functional~\eq{l} is equivalent to the minimization problem for the energy functional
\be{J}
J = \frac 12 \int_0^{t_1} (u_1^2 + u_2^2) \, d t \to \min
\ee
with fixed $t_1$.


System~\eq{sys2} has full rank:
\begin{align}
&\xi_3 = [\xi_1, \xi_2] = \sin \theta \pder{}{x} - \cos \theta \pder{}{y}, \label{xi3} \\
&\spann(\xi_1(q), \xi_2(q), \xi_3(q)) = T_q M \qquad \forall \ q \in M, \label{spanxi123}
\end{align}
so it is completely controllable on $M$. 

Another standard reasoning proves existence of solutions to optimal control problem~\eq{sys2},  \eq{bound}, \eq{J}. First the problem is equivalently reduced to the time-optimal problem with dynamics~\eq{sys2}, boundary conditions~\eq{bound}, restrictions on control $u_1^2 + u_2^2 \leq 1$, and the cost functional $t_1 \to \min$. Then the state space of the problem is embedded into $\R^3$ (see e.g.~\cite{el_max}), and finally Filippov's theorem~\cite{notes} implies existence 
 of optimal controls. 

\section{Pontryagin Maximum Principle}
\label{sec:PMP}

We apply the version of PMP adapted to left-invariant optimal control problems and use the basic notions of the Hamiltonian formalism as described in~\cite{notes}. In particular, we denote by $T^*M$ the cotangent bundle of a manifold $M$, by $\map{\pi}{T^*M}{M}$ the canonical projection, and by $\vh\in \Vec(T^*M)$ the Hamiltonian vector field corresponding to a Hamiltonian function $h \in \Cinft$.

Consider linear on fibers Hamiltonians corresponding to the vector fields $\xi_i$:
$$
h_i(\lam) = \lan \lam, \xi_i(q)\ran, \quad q = \pi(\lam), \quad \lam \in T^*M, \quad i = 1, 2, 3,
$$
and the control-dependent Hamiltonian of PMP
$$
h_u^{\nu}(\lam) = \frac{\nu}{2}(u_1^2 + u_2^2) + u_1 h_1(\lam) + u_2 h_2(\lam), 
\quad \lam \in T^*M, \quad u \in \R^2, \quad \nu \in \{ -1, 0 \}.
$$

Then the Pontryagin Maximum Principle~\cite{PBGM, notes} for the problem under consideration reads as follows.

\begin{theorem}
\label{th:PMP}
Let $u(t)$ and $q(t)$, $t \in [0, t_1]$, be an optimal control and the corresponding optimal trajectory in problem~\eq{sys2}, \eq{bound}, \eq{J}. Then there exist a Lipschitzian curve $\lam_t \in T^*M$, $\pi(\lam_t) = q(t)$, $t \in [0, t_1]$, and a number $\nu \in \{-1, 0\}$ for which the following conditions hold for almost all $t \in [0,t_1]$:
\begin{align}
&\dot \lam_t = \vh^{\nu}_{u(t)}(\lam_t) = u_1(t) \vh_1(\lam_t) + u_2(t) \vh_2(\lam_t), \label{PMP1} \\
&h^{\nu}_{u(t)}(\lam_t) = \max_{u \in \R^2} h^{\nu}_{u}(\lam_t),   \label{PMP2} \\
&(\nu, \lam_t) \neq 0. \label{PMP3}
\end{align}
\end{theorem}

Relations~\eq{xi3}, \eq{spanxi123} mean that the sub-Riemannian problem under consideration is contact, thus in the abnormal case $\nu = 0$ the optimal trajectories are constant.

Consider now the normal case $\nu = -1$. Then the maximality condition~\eq{PMP2} implies that normal extremals satisfy the equalities 
$$
u_i(t) = h_i(\lam_t), \qquad i = 1, 2,
$$
thus they are trajectories of the normal Hamiltonian system
\be{dlamvH}
\dot\lam = \vH(\lam), \qquad \lam \in T^*M,
\ee
with the maximized Hamiltonian $H = (h_1^2 + h_2^2)/2$.

In view of the multiplication table
$$
[\xi_1, \xi_2] = \xi_3, \quad  
[\xi_1, \xi_3] = 0, \quad
[\xi_2, \xi_3] = \xi_1,
$$
system~\eq{dlamvH} reads in coordinates as follows:
\begin{align}
&\dot h_1 = - h_2 h_3, \quad \dot h_2 = h_1 h_3, \quad \dot h_3 = h_2 h_3,     \label{dh123} \\
&\dot x = h_1 \cos \theta, \quad \dot y = h_1 \sin \theta, \quad \dot \theta = h_2.  \label{dxyth}
\end{align}
Along all normal extremals we have $H \equiv C \geq 0$; moreover, for non-constant normal extremal trajectories $C>0$. Since the normal Hamiltonian system~\eq{dh123}, \eq{dxyth} is homogeneous w.r.t. $(h_1, h_2)$, we can consider its trajectories only on the level surface $H = 1/2$ (this corresponds to the arc-length parametrization of extremal trajectories), and set the terminal time $t_1$ free. Then the initial covector $\lam$ for normal extremals $\lam_t = e^{t \vH}(\lam)$ belongs to the initial cylinder
$$
C = T_{q_0}^*M \cap \{ H(\lam) = 1/2 \}.
$$
Introduce the polar coordinates
$$
h_1 = \cos \a, \quad h_2 = \sin \a,
$$
then the initial cylinder decomposes as $C \cong S^1_{\a} \times \R_{h_3}$, where $S^1_{\a} = \R /(2 \pi \Z)$. In these coordinates the vertical part~\eq{dh123} reads as
\be{dah3}
\dot \a = h_3, \quad \dot h_3 = \frac 12 \sin 2 \a, \quad (\a, h_3) \in C.
\ee
In the coordinates
$$
\g = 2 \a + \pi \in 2 S^1 = \R /(4 \pi \Z), \qquad c = 2 h_3 \in \R,
$$
system~\eq{dah3} takes the form of the standard pendulum
\be{ham_vert}
\dot \g = c, \quad \dot c = -\sin \g, \qquad (\g, c) \in C \cong (2 S^1_{\g}) \times \R_c.
\ee
Here $2 S^1 = \R /(4 \pi \Z)$ is   the double covering  of the standard circle $S^1 = \R /(2 \pi \Z)$.
Then the  horizontal part~\eq{dxyth} of the normal Hamiltonian system reads as
\be{ham_hor}
\dot x = \sin \frac{\g}{2} \cos \theta, \quad \dot y = \sin \frac{\g}{2} \sin \theta, \quad \dot \theta = - \cos \frac{\g}{2}.
\ee

Summing up, all nonconstant arc-length parametrized optimal trajectories in the sub-Riemannian problem on the Lie group $\SE(2)$ are projections of solutions to the normal Hamiltonian system~\eq{ham_vert}, \eq{ham_hor}. 

\section{Exponential mapping}
\label{sec:exp}
The family of arc-length parametrized normal extremal trajectories is described by the exponential mapping
\begin{align*}
&\map{\Exp}{N}{M}, \qquad N = C \times \R_+, \\
&\Exp(\nu) = \Exp(\lam,t) = \pi \circ e^{t \vH}(\lam) = \pi(\lam_t) = q(t), \\
&\nu = (\lam,t) = (\g, c, t) \in N.
\end{align*}
In this section we derive explicit formulas for the exponential mapping in special elliptic coordinates in $C$ induced by the flow of the pendulum~\eq{ham_vert}. The general construction of elliptic coordinates was developed in~\cite{dido_exp, max1, el_max}, here they are adapted to the problem under consideration.

\subsection{Decomposition of the cylinder $C$}
The equation of  pendulum~\eq{ham_vert} has the energy integral
\be{E}
E = \frac{c^2}{2} - \cos \g \in [-1, + \infty).
\ee
Consider the following decomposition of the cylinder $C$ into disjoint invariant sets of the pendulum:
\begin{align}
&C = \bigcup_{i=1}^5 C_i, \label{decompC} \\
&C_1 = \{ \lam \in C \mid E \in (-1, 1) \}, \nonumber \\
&C_2 = \{ \lam \in C \mid E \in (1, + \infty) \}, \nonumber \\
&C_3 = \{ \lam \in C \mid E =1, \ c \neq 0 \}, \nonumber \\
&C_4 = \{ \lam \in C \mid E = - 1 \} = \{ (\g, c) \in C \mid \g = 2 \pi n, \ c = 0 \}, \nonumber \\
&C_5 = \{ \lam \in C \mid E = 1, \ c = 0 \} = \{ (\g, c) \in C \mid \g = \pi + 2 \pi n, \ c = 0 \}. \nonumber 
\end{align}
Here and below we denote by $n$ a natural number.

Denote the connected components of the sets $C_i$:
\begin{align*}
&C_1 = \cup_{i=0}^1 C_1^i, \qquad
C_1^i = \{(\g,c) \in C_1 \mid \sgn(\cos (\g/2)) = (-1)^i\}, \quad i = 0, 1, \\
&C_2 = C_2^+ \cup C_2^-, \qquad
C_2^{\pm} = \{(\g,c) \in C_2 \mid \sgn c =  \pm 1 \}, \\
&C_3 = \cup_{i=0}^1 (C_3^{i+} \cup C_3^{i-}), \\
&\qquad\qquad
C_3^{i\pm} = \{(\g,c) \in C_3 \mid \sgn(\cos (\g/2)) = (-1)^i, \ \sgn c =  \pm 1 \},  \quad i = 0, 1,\\
&C_4 = \cup_{i=0}^1 C_4^i, \qquad
C_4^i = \{(\g,c) \in C  \mid \g = 2 \pi i, \ c = 0\}, \quad i = 0, 1, \\
&C_5 = \cup_{i=0}^1 C_5^i, \qquad
C_5^i = \{(\g,c) \in C  \mid \g = \pi + 2 \pi i, \ c = 0\}, \quad i = 0, 1.
\end{align*}

Decomposition~\eq{decompC} of the cylinder $C$ is shown at Fig.~\ref{fig:C_decomp}.

\onefiglabelsize{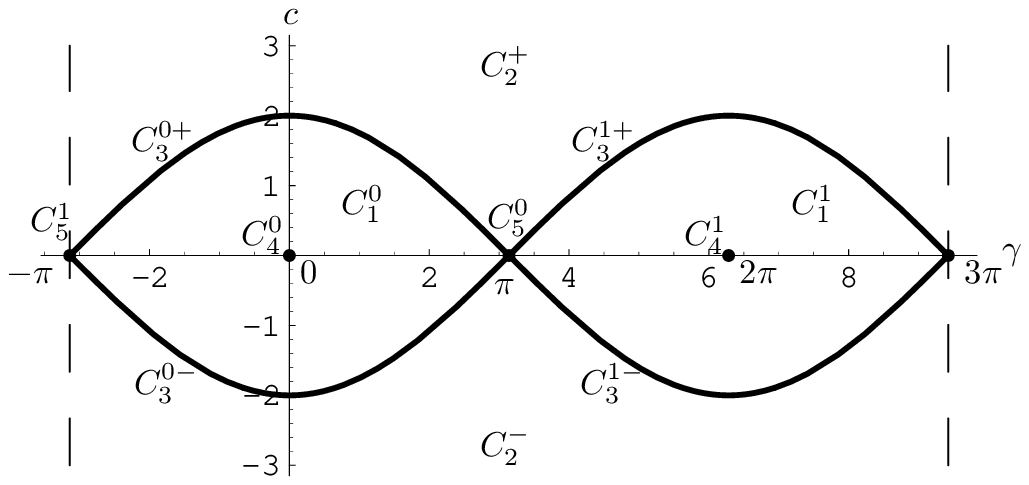}{Decomposition of the cylinder $C$}{fig:C_decomp}{1}

\subsection{Elliptic coordinates on the cylinder $C$}
\label{subsec:ell_coordsC}
According to the general construction developed in~\cite{el_max}, we introduce elliptic coordinates $(\f,k)$ on the domain $C_1 \cup C_2 \cup C_3$ of the cylinder $C$, where $k$ is a reparametrized energy, and $\f$ is the time of motion of the pendulum~\eq{ham_vert}. We use Jacobi's functions $\am(\f,k)$, $\cn(\f,k)$, $\sn(\f,k)$, $\dn(\f,k)$, $\E(\f,k)$; moreover, $K(k)$ is the complete elliptic integral of the first kind~\cite{whit_watson}.

If $\lam = (\g,c) \in C_1$, then:
\begin{align*}
&k = \sqrt{\frac{E+1}{2}} = \sqrt{\sin^2 \frac{\g}{2} + \frac{c^2}{4}} \in (0,1),\\
&\sin \frac{\g}{2} = s_1 k \sn(\f,k), \qquad s_1 = \sgn \cos(\g/2),\\
&\cos \frac{\g}{2} = s_1 \dn(\f,k), \\
&\frac{c}{2} = k \cn(\f,k), \qquad \f \in [0, 4 K(k)].
\end{align*}

If $\lam = (\g,c) \in C_2$, then:
\begin{align*}
&k = \sqrt{\frac{2}{E+1}} = \frac{1}{\sqrt{\sin^2 \frac{\g}{2} + \frac{c^2}{4}}} \in (0,1),\\
&\sin \frac{\g}{2} = s_2  \sn(\f/k,k), \qquad s_2 = \sgn c, \\
&\cos \frac{\g}{2} = \cn(\f/k,k), \\
&\frac{c}{2} = (s_2/ k) \dn(\f/k,k), \qquad \f \in [0, 4 k K(k)].
\end{align*}

If $\lam = (\g,c) \in C_3$, then:
\begin{align*}
&k = 1,\\
&\sin \frac{\g}{2} = s_1 s_2 \tanh \f,   \qquad s_1 = \sgn \cos(\g/2), \quad s_2 = \sgn c, \\
&\cos \frac{\g}{2} = s_1 / \cosh \f, \\
&\frac{c}{2} = s_2/ \cosh \f, \qquad \f \in (-\infty, +\infty).
\end{align*}

\subsection{Parametrization of extremal trajectories}
\label{subsec:param_extr}
In the elliptic coordinates the flow of the pendulum~\eq{ham_vert} rectifies:
$$
\dot \f = 1, \quad \dot k = 0, \qquad \lam = (\f, k) \in \cup_{i=1}^3 C_i,
$$
this is verified directly using the formulas of Subsec.~\ref{subsec:ell_coordsC}. Thus the vertical subsystem of the normal Hamiltonian system of PMP~\eq{ham_vert} is trivially integrated: one should just substitute $\f_t = \f + t$, $k \equiv \const$ to the formulas of elliptic coordinates of Subsec.~\ref{subsec:ell_coordsC}. Integrating the horizontal subsystem~\eq{ham_hor}, we obtain the following parametrization of extremal trajectories.

If $\lam = (\f,k) \in C_1$, then $\f_t = \f + t$ and:
\begin{align*}
&\cos \theta_t = \cn \f \cn \f_t + \sn \f \sn \f_t, \\ 
&\sin \theta_t = s_1(\sn \f \cn \f_t - \cn \f \sn \f_t), \\
&\theta_t = s_1(\am \f - \am  \f_t) \pmod {2 \pi}, \\
&x_t = (s_1/k) [ \cn \f (\dn \f - \dn \f_t) + \sn \f (t + \E(\f) - \E(\f_t))], \\  
&y_t = (1/k) [ \sn \f (\dn \f - \dn \f_t) - \cn \f (t + \E(\f) - \E(\f_t))].
\end{align*}

In the domain $C_2$, it will be convenient to use the coordinate 
$$
\psi = \f/k, \qquad \psi_t = \f_t/k = \psi + t/k.
$$
If $\lam \in C_2$, then:
\begin{align*}
&\cos \theta_t = k^2 \sn \psi \sn \psi_t + \dn \psi \dn \psi_t, \\ 
&\sin \theta_t = k(\sn \psi \dn \psi_t - \dn \psi \sn \psi_t), \\
&x_t = s_2 k [\dn \psi(\cn \p - \cn \p_t) + \sn \p (t/k + \E(\p) - \E(\p_t))], \\
&y_t = s_2  [k^2 \sn \psi (\cn \p - \cn \p_t) - \dn \p (t/k + \E(\p) - \E(\p_t))].
\end{align*}

If $\lam \in C_3$, then:
\begin{align*}
&\cos \theta_t = 1/ (\cosh \f \cosh \f_t)  + \tanh \f  \tanh \f_t, \\
&\sin \theta_t = s_1 (\tanh \f /\cosh \f_t - \tanh \f_t /\cosh \f), \\
&x_t = s_1 s_2 [(1/\cosh \f)(1/\cosh \f - 1/\cosh \f_t) + \tanh \f(t + \tanh \f - \tanh \f_t)],\\
&y_t = s_2 [\tanh \f (1/\cosh \f - 1/\cosh \f_t) -(1/\cosh \f) (t + \tanh \f - \tanh \f_t)].
\end{align*}

In the degenerate cases, the normal Hamiltonian system~\eq{ham_vert}, \eq{ham_hor} is easily integrated.

If $\lam \in C_4$, then:
$$
\theta_t = -s_1 t, \qquad
x_t = 0, \qquad y_t = 0.
$$

If $\lam \in C_5$, then:
$$
\theta_t = 0, \qquad
x_t = t \, \sgn \sin (\g/2), \qquad  y_t = 0.
$$

It is easy to compute from the Hamiltonian system~\eq{ham_vert}, \eq{ham_hor} that projections $(x_t,y_t)$ of extremal trajectories have curvature $\kappa = - \cot (\g_t/2)$.  Thus they have inflection points when $\cos(\g_t/2) = 0$, and cusps when $\sin(\g_t/2) = 0$. Each curve $(x_t,y_t)$ for $\lam \in \cup_{i=1}^3 C_i$ has cusps. In the case  $\lam \in C_1 \cup C_3$ these curves have no inflection points, and in the case  $\lam \in C_2$ each such curve has inflection points. Plots of the curves $(x_t,y_t)$ in the cases $\lam \in C_1 \cup C_2 \cup C_3$ are given respectively at Figs.~\ref{fig:xyC1}, \ref{fig:xyC2}, \ref{fig:xyC3}. 

\twofiglabel
{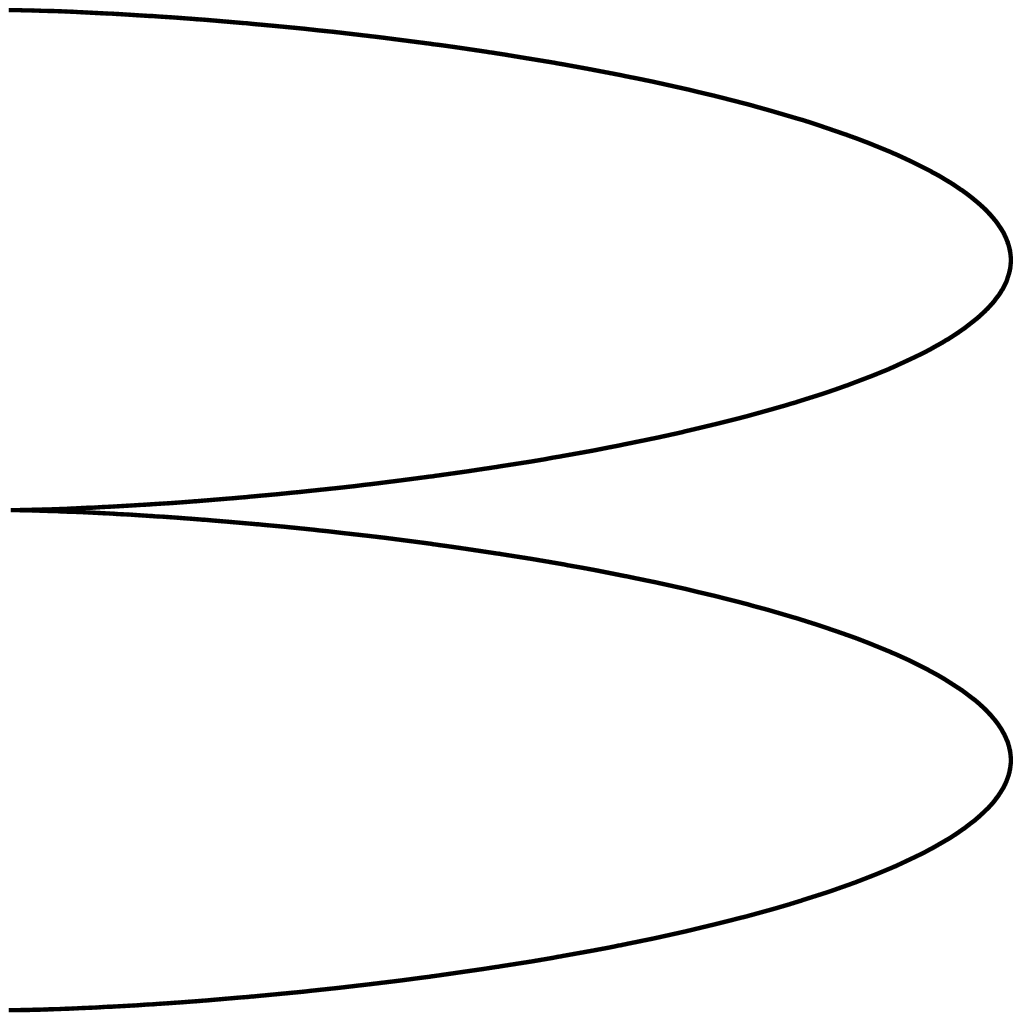}{Non-inflexional trajectory: $\lam \in C_1$}{fig:xyC1}
{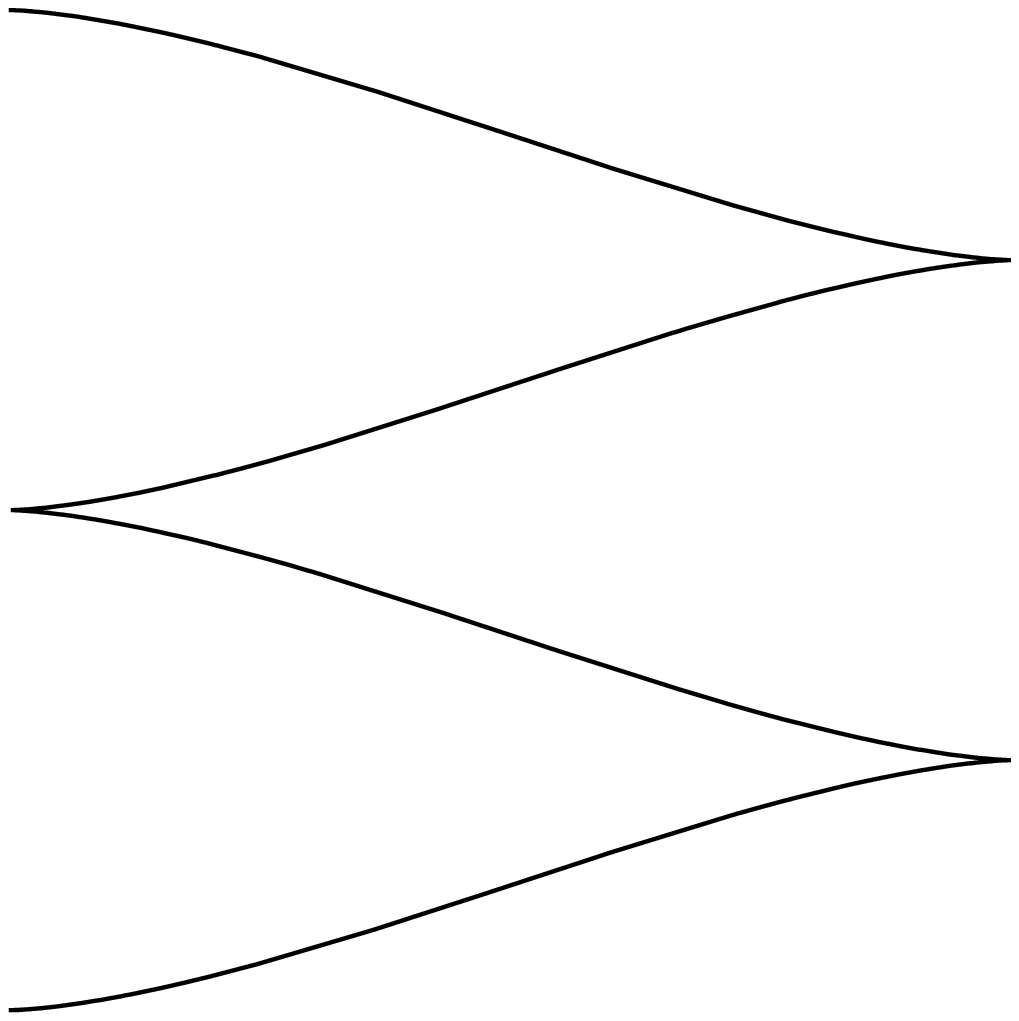}{Inflexional trajectory: $\lam \in C_2$}{fig:xyC2}

\onefiglabel
{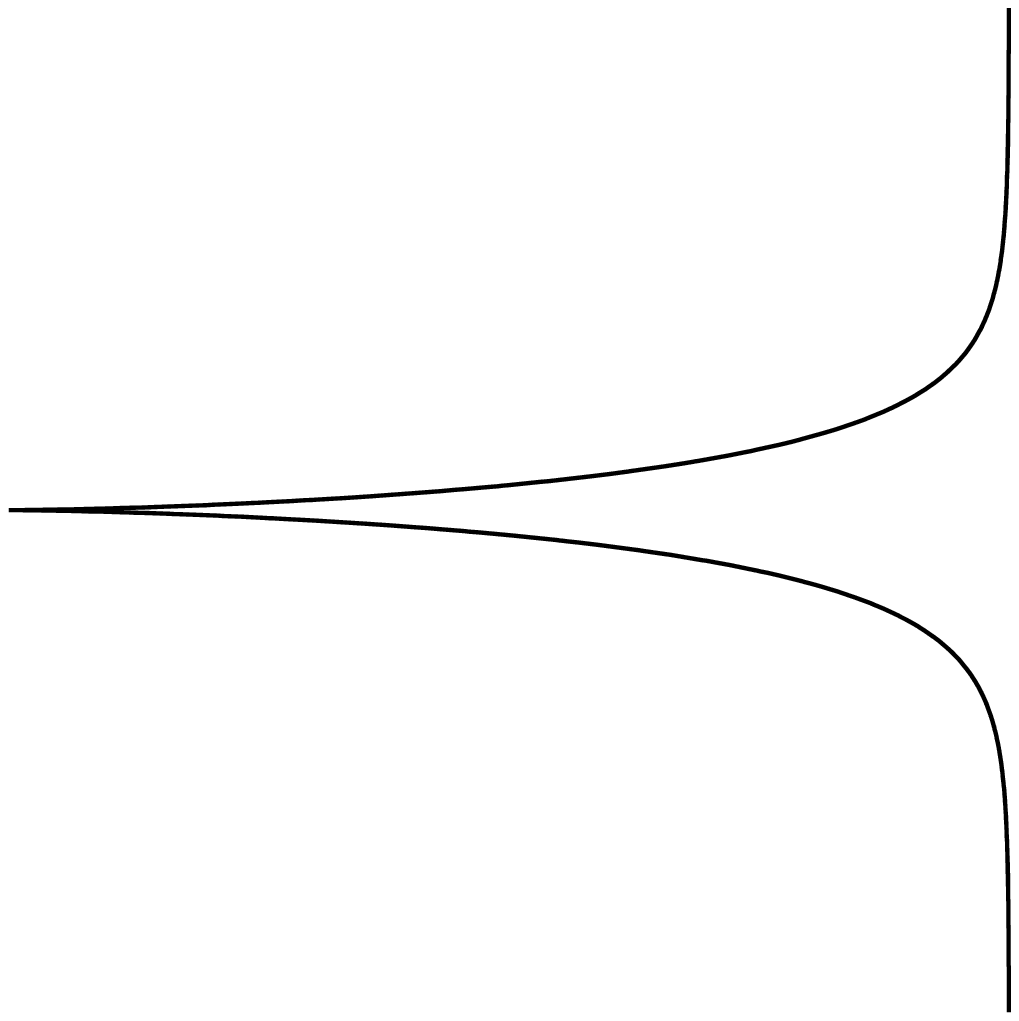}{Critical trajectory: $\lam \in C_3$}{fig:xyC3}

In the cases $\lam \in C_4$ and $\lam \in C_5$ the extremal trajectories $q_t$ are respectively Riemannian geodesics in the circle $\{ x = y = 0\}$ and  in the plane $\{ \theta = 0 \}$.

\section{Discrete symmetries and Maxwell strata}
\label{sec:max}
In this section we continue reflections in the state cylinder of the standard pendulum to discrete symmetries of the exponential mapping.

\subsection
{Symmetries of the vertical part  of  Hamiltonian system}
\label{subsec:sym_vert}

\subsubsection{Reflections in the state cylinder of pendulum}

The phase portrait of pendulum~\eq{ham_vert} admits the following reflections:
\begin{align*}
&\map{\eps^1}{(\g,c)}{(\g,-c)},\\
&\map{\eps^2}{(\g,c)}{(-\g,c)},\\
&\map{\eps^3}{(\g,c)}{(-\g,-c)},\\
&\map{\eps^4}{(\g,c)}{(\g+2 \pi,c)},\\
&\map{\eps^5}{(\g,c)}{(\g + 2 \pi, -c)},\\
&\map{\eps^6}{(\g,c)}{(-\g + 2 \pi,c)},\\
&\map{\eps^7}{(\g,c)}{(-\g + 2 \pi,-c)}.
\end{align*}
These reflections generate the group of symmetries of a parallelepiped $G = \{ \Id, \eps^1, \dots, \eps^7\}$. The reflections $\eps^3$, $\eps^4$, $\eps^7$ preserve direction of time on trajectories of pendulum, while the reflections $\eps^1$, $\eps^2$, $\eps^5$, $\eps^6$ reverse the direction of time.

\subsubsection{Reflections of trajectories of pendulum}
\begin{proposition}
\label{propos:refl_gc}
The following mappings transform trajectories of pendulum~\eq{ham_vert} to trajectories:
\be{epsi_delta}
\mapto{\eps^i}{\delta = \{(\g_s,c_s)\mid s \in [0,t]\}}{\delta^i = \{(\g_s^i,c_s^i)\mid s \in [0,t]\}}, \quad i = 1, \dots, 7,
\ee
where
\begin{align*}
&(\g_s^1, c_s^1) = (\g_{t-s}, - c_{t-s}), \\
&(\g_s^2, c_s^2) = (-\g_{t-s}, c_{t-s}), \\
&(\g_s^3, c_s^3) = (-\g_{s}, - c_{s}), \\
&(\g_s^4, c_s^4) = (\g_{s} + 2 \pi, c_{s}), \\
&(\g_s^5, c_s^5) = (\g_{t-s} + 2 \pi, - c_{t-s}), \\
&(\g_s^6, c_s^6) = (-\g_{t-s} + 2 \pi, c_{t-s}), \\
&(\g_s^7, c_s^7) = (-\g_{s}+ 2 \pi, - c_{s}).
\end{align*}
\end{proposition}
\begin{proof}
The statement is verified by substitution to system~\eq{ham_vert} and differentiation.
\end{proof} 

The action~\eq{epsi_delta} of reflections $\eps^i$ on trajectories $\delta$ of the pendulum~\eq{ham_vert} is illustrated at Fig.~\ref{fig:eps_gc}.

\onefiglabelsize{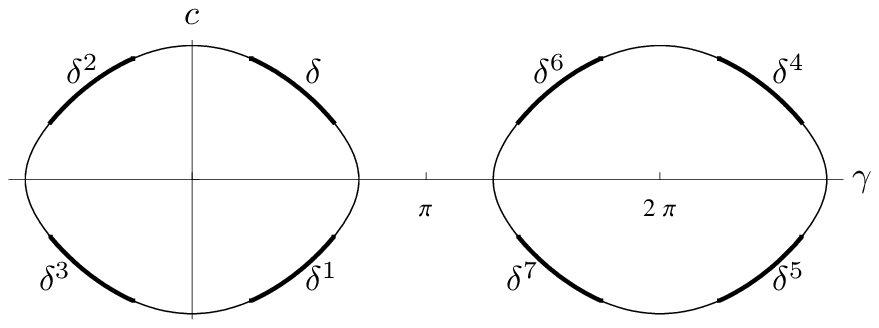}{Reflections $\mapto{\eps^i}{\delta}{\delta^i}$ of trajectories of pendulum}{fig:eps_gc}{1}

\subsection{Symmetries of Hamiltonian system}
\subsubsection{Reflections of extremals}
We define action of the group $G$ on the normal extremals $\lam_s = e^{s \vH}(\lam_0) \in T^* M$, $s \in [0,t]$, i.e., solutions to the normal Hamiltonian system
\begin{align}
&\dot \g_s = c_s, \qquad \dot c_s = - \sin \g_s, \label{Ham_s1}\\
&\dot q_s = \sin \frac{\g_s}{2} \, X_1(q_s) - \cos \frac{\g_s}{2}\,  X_2(q_s) \label{Ham_s2}
\end{align}
as follows:
\begin{align}
&\mapto{\eps^i}{\{\lam_s \mid s \in [0,t]\}}{\{\lam_s^i\mid s \in [0,t]\}}, \qquad i = 1, \dots, 7, \label{epsilam}\\
&\lam_s = (\g_s, c_s, q_s), \qquad \lam_s^i = (\g_s^i, c_s^i, q_s^i). \label{lamsgams}
\end{align}
Here $\lam_s^i$ is a solution to the Hamiltonian system~\eq{Ham_s1}, \eq{Ham_s2}, and the action of reflections on the vertical coordinates $(\g_s,c_s)$ was defined in Subsec.~\ref{subsec:sym_vert}.
The action of reflections on the horizontal coordinates $(x_s,y_s,\theta_s)$ is described as follows.

\begin{proposition}
\label{propos:refl_xyth}
Let $q_s = (x_s, y_s, \theta_s)$, $s \in [0, t]$, be a normal extremal trajectory, and let $q_s^i = (x_s^i, y_s^i, \theta_s^i)$, $s \in [0, t]$, be its image under the action of the reflection $\eps^i$ as defined by~\eq{epsilam}, \eq{lamsgams}. Then the following equalities hold:
\begin{align*}
&(1) && \theta_1^s = \theta_t - \theta_{t-s}, \\
& && x^1_s = \cos \theta_t (x_t - x_{t-s}) + \sin \theta_t (y_t - y_{t-s}), \\
& && y^1_s = \sin \theta_t (x_t - x_{t-s}) - \cos \theta_t (y_t - y_{t-s}),\\
&(2) && \theta_2^s = \theta_t - \theta_{t-s}, \\
& && x^2_s = -\cos \theta_t (x_t - x_{t-s}) - \sin \theta_t (y_t - y_{t-s}), \\
& && y^2_s = -\sin \theta_t (x_t - x_{t-s}) + \cos \theta_t (y_t - y_{t-s}),\\
&(3) && \theta_3^s = \theta_s, \\
& && x^3_s = - x_s, \\
& && y^3_s = -y_s,  \\
&(4) && \theta_4^s = -\theta_s, \\
& && x^4_s = - x_s, \\
& && y^4_s = y_s,  \\
&(5) && \theta_5^s = \theta_{t-s} -  \theta_{t}, \\
& && x^5_s = \cos \theta_t (x_{t-s} - x_{t}) + \sin \theta_t (y_{t-s} - y_{t}), \\
& && y^5_s = -\sin \theta_t (x_{t-s} - x_{t}) + \cos \theta_t (y_{t-s} - y_{t}),\\
&(6) && \theta_6^s = \theta_{t-s} - \theta_{t}, \\
& && x^6_s = \cos \theta_t (x_t - x_{t-s}) + \sin \theta_t (y_t - y_{t-s}), \\
& && y^6_s = -\sin \theta_t (x_t - x_{t-s}) + \cos \theta_t (y_t - y_{t-s}),\\
&(7) && \theta_7^s = -\theta_s, \\
& && x^7_s = x_s, \\
& && y^7_s = -y_s.
\end{align*}
\end{proposition}
\begin{proof}
We prove only the formulas for $\theta^1_s$ and $x^1_s$ since all other equalities are proved similarly.

By Proposition~\ref{propos:refl_gc}, we have $\g_s^1 = \g_{t-s}$. Then we obtain from~\eq{ham_vert}:
$$
\theta^1_s = \int_0^s -\cos \frac{\g^1_r}{2} \, dr = - \int_0^s \cos \frac{\g_{t-r}}{2} \, dr = 
\int_t^{t-s} \cos \frac{\g_{p}}{2} \, dp = \theta_t - \theta_{t-s}
$$
and
\begin{align*}
x^1_s &= \int_0^s \sin \frac{\g^1_r}{2} \cos \theta^1_r \, dr =  
\int_0^s \sin \frac{\g_{t-r}}{2} \cos (\theta_t - \theta_{t-r}) \, dr\\
&=  -\cos \theta_t \int_t^{t-s} \sin \frac{\g_{p}}{2} \cos \theta_p \, dp - 
 \sin \theta_t \int_t^{t-s} \sin \frac{\g_{p}}{2} \sin \theta_p \, dp\\
 &= \cos \theta_t (x_t - x_{t-s}) + \sin \theta_t (y_t - y_{t-s}).
\end{align*}
\end{proof}

The action of reflections $\eps^i$ on curves $(x_s,y_s)$ has a simple visual meaning. Up to rotations of the plane $(x,y)$, the mappings $\eps^1$, $\eps^2$, $\eps^3$ are respectively reflections of the curves $\{(x_s,y_s) \mid s \in [0,t]\}$ in the center of the segment $l$ connecting the endpoints $(x_0, y_0)$ and $(x_t, y_t)$, in the middle perpendicular to $l$, and in $l$ itself (see~\cite{max1, el_max}). The mapping $\eps^4$ is  the reflection in the axis $y$ perpendicular to the initial velocity vector $(\cos \theta_0, \sin \theta_0)$. The rest mappings are represented   as follows: $\eps^{i+4} = \eps^4 \circ \eps^i$, $i = 1, 2, 3$.

\subsubsection{Reflections of endpoints of extremal trajectories}

We define action of reflections in the state space $M$ as the action on endpoints of extremal trajectories
\be{epsiq}
\map{\eps^i}{M}{M}, \qquad \mapto{\eps^i}{q_t}{q^i_t},
\ee
see~\eq{epsilam}, \eq{lamsgams}.
By virtue of Propos.~\ref{propos:refl_xyth}, the point $q_t^i$ depends only on the endpoint $q_t$, not on the whole trajectory $\{q_s \mid s \in [0,t]\}$. 

\begin{proposition}
\label{propos:qit}
Let $q = (x,y,\theta) \in M$, $q^i = \eps^i(q) = (x^i, y^i, \theta^i) \in M$. Then:
\begin{align*}
&(x^1,y^1,\theta^1) = (x \cos \theta + y \sin \theta, x \sin \theta - y \cos \theta, \theta), \\
&(x^2,y^2,\theta^2) = (-x \cos \theta - y \sin \theta, -x \sin \theta + y \cos \theta, \theta), \\
&(x^3,y^3,\theta^3) = (-x, -y, \theta), \\
&(x^4,y^4,\theta^4) = (-x, y, -\theta), \\
&(x^5,y^5,\theta^5) = (-x \cos \theta - y \sin \theta, x \sin \theta - y \cos \theta, -\theta), \\
&(x^6,y^6,\theta^6) = (x \cos \theta + y \sin \theta, -x \sin \theta + y \cos \theta, -\theta), \\
&(x^7,y^7,\theta^7) = (x, -y, -\theta).
\end{align*}
\end{proposition}
\begin{proof}
It suffices to  substitute $s=0$ to the formulas of Proposition~\ref{propos:refl_xyth}.
\end{proof}

\subsection{Reflections as symmetries of exponential mapping}

Define action of the reflections in the preimage of the exponential mapping:
\be{epsinu}
\map{\eps^i}{N}{N}, \qquad
\mapto{\eps^i}{\nu = (\g,c,t)}{\nu^i = (\g^i, c^i,t)},
\ee
where $(\g, c) = (\g_0, c_0)$ and $(\g^i, c^i) = (\g^i_0, c^i_0)$ are the initial points of the corresponding trajectories of pendulum $(\g_s, c_s)$ and $(\g_s^i, c_s^i)$. The explicit formulas for $(\g^i,c^i)$ are given by the following statement.

\begin{proposition}
\label{propos:lami}
Let $\nu = (\lam,t)  = (\g,c,t) \in N$, $\nu^i = \eps^i(\nu) = (\lam^i,t) = (\g^i,c^i,t)\in N$. Then:
\begin{align*}
&(\g^1,c^1) = (\g_t, -c_t), \\
&(\g^2,c^2) = (-\g_t, c_t), \\
&(\g^3,c^3) = (-\g, -c), \\
&(\g^4,c^4) = (\g + 2 \pi, c), \\
&(\g^5,c^5) = (\g_t + 2 \pi, -c_t), \\
&(\g^6,c^6) = (-\g_t+ 2 \pi, c_t), \\
&(\g^7,c^7) = (-\g, -c).
\end{align*}
\end{proposition}
\begin{proof}
Apply Proposition~\ref{propos:refl_gc} with $s = 0$.
\end{proof}

Formulas~\eq{epsiq}, \eq{epsinu} define the action of reflections $\eps^i$ in the image and preimage of the exponential mapping. Since the both actions of $\eps^i$ in $M$ and $N$ are induced by the action of $\eps^i$ on extremals $\lam_s$~\eq{epsilam}, we obtain the following statement.

\begin{proposition}
For any $i = 1, \dots, 7$, the reflection $\eps^i$ is a symmetry of the exponential mapping, i.e., the following diagram is commutative:
$$
\xymatrix{
N \ar[r]^{\Exp} \ar[d]^{\eps^i} & M  \ar[d]^{\eps^i}\\
N \ar[r]^{\Exp} & M
}
\qquad\qquad\qquad
\xymatrix{
\nu \ar@{|->}[r]^{\Exp} \ar@{|->}[d]^{\eps^i} & q  \ar@{|->}[d]^{\eps^i}\\
\nu^i  \ar@{|->}[r]^{\Exp} & q^i
}
$$
\end{proposition}

\section{Maxwell strata corresponding to reflections}
\label{sec:estim}

\subsection
[Maxwell points and optimality of extremal trajectories]
{Maxwell points and optimality \\ of extremal trajectories}
A point $q_t$ of a sub-Riemannian geodesic is called a Maxwell point if there exists another extremal trajectory $\tq_s \not\equiv q_s$ such that $\tq_t = q_t$ for the instant of time $t> 0$. It is well known that after a Maxwell point a sub-Riemannian geodesic cannot be optimal (provided the problem is analytic). 

In this section we compute Maxwell points corresponding to reflections. For any $i = 1, \dots, 7$, define the Maxwell stratum in the preimage of the exponential mapping corresponding to the reflection $\eps^i$ as follows:
\be{MAX_def}
\MAX^i = \{\nu = (\lam,t) \in N \mid \lam \neq \lam^i, \ \Exp(\lam,t) = \Exp(\lam^i,t)\}.
\ee
We denote the corresponding Maxwell stratum in the image of the exponential mapping as
$$
\Max^i = \Exp(\MAX^i) \subset M.
$$
If $\nu = (\lam,t) \in \MAX^i$, then $q_t = \Exp(\nu)\in \Max^i$ is a Maxwell point along the trajectory $q_s = \Exp(\lam,s)$. Here we use the fact that if $\lam \neq \lam^i$, then $\Exp(\lam,s) \not\equiv \Exp(\lam^i,s)$. 

\subsection{Multiple points of exponential mapping}
In this subsection we study solutions to the equation $q = q^i$, where $q^i = \eps^i(q)$, that appears in  definition~\eq{MAX_def} of Maxwell strata $\MAX^i$.

The following functions are defined on $M = \R^2_{x,y} \times S^1_{\theta}$ up to sign:
$$
R_1 = y \cos \frac{\theta}{2} - x \sin \frac{\theta}{2}, \qquad 
R_2 = x \cos \frac{\theta}{2} + x \sin \frac{\theta}{2},
$$
although their zero sets $\{ R_i = 0 \}$ are well-defined. In the polar coordinates
$$
x = \rho \cos \chi, \qquad y = \rho \sin \chi,
$$
these functions read as
$$
R_1 = \rho \sin\left(\chi - \frac{\theta}{2}\right), \qquad
R_2 = \rho \cos\left(\chi - \frac{\theta}{2}\right).
$$

\begin{proposition}
\label{propos:qt=qit}
\begin{itemize}
\item[$(1)$] 
$q^1 = q \iff R_1(q) = 0.$
\item[$(2)$] 
$q^2 = q \iff R_2(q) = 0.$
\item[$(3)$] 
$q^3 = q \iff x = y  = 0.$
\item[$(4)$] 
$q^4 = q \iff \sin \theta = x  = 0.$
\item[$(5)$] 
$q^5 = q \iff \theta = \pi \text{ or } (x,y,\theta) = (0,0,0).$
\item[$(6)$] 
$q^6 = q \iff \theta  = 0.$
\item[$(7)$] 
$q^7 = q \iff \sin \theta = y  = 0.$
\end{itemize}
\end{proposition}
\begin{proof}
We prove only item (1), all the rest items are considered similarly. By virtue of Proposition~\ref{propos:qit}, we have
\begin{align*}
q^1 = q &\iff
\begin{cases}
x \cos \theta + y \sin \theta = x \\
x \sin \theta - y \cos \theta = y
\end{cases} 
\iff
\begin{cases}
\rho \sin \frac{\theta}{2} \sin\left(\chi - \frac{\theta}{2}\right) = 0 \\
\rho \cos \frac{\theta}{2} \sin\left(\chi - \frac{\theta}{2}\right) = 0
\end{cases}\\
&\iff
\rho \sin\left(\chi - \frac{\theta}{2}\right) \iff 
R_1(q) = 0.
\end{align*}
\end{proof}

Proposition~\ref{propos:qt=qit} implies that all Maxwell strata corresponding to reflections  satisfy the inclusion
$$
\Max^i \subset \{ q \in M \mid R_1(q) R_2(q) \sin \theta = 0 \}.
$$
The equations $R_i(q) = 0$, $i = 1, 2$, define two Moebius strips, while the equation $\sin \theta = 0$ determines two discs in the state space $M = \R^2_{x,y} \times S^1_{\theta}$, see Fig.~\ref{fig:R1R2sinth=0}.

\onefiglabelsize{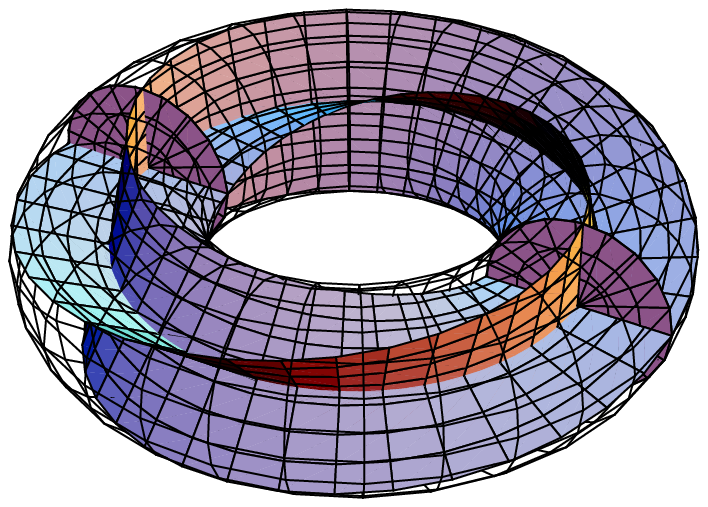}{Surfaces containing Maxwell strata $\Max^i$}{fig:R1R2sinth=0}{1}

By virtue of Propos.~\ref{propos:qt=qit}, the Maxwell strata $\Max^3$, $\Max^4$, $\Max^7$ are one-dimensional and are contained in the two-dimensional strata $\Max^1$, $\Max^2$, $\Max^5$, $\Max^6$. Thus in the sequel we restrict ourselves only by the 2-dimensional strata.

\subsection
[Fixed points of reflections in  preimage of  exponential mapping]
{Fixed points of reflections \\ in  preimage of  exponential mapping}
In this subsection we describe solutions to the equations $\lam = \lam^1$ essential for explicit characterization of the Maxwell strata $\MAX^i$, see~\eq{MAX_def}.

From now on we will widely use the following variables in the sets $N_i$, $i = 1, 2, 3$:
\begin{align*}
&\nu = (\lam,t) \in N_1 \then &&\tau = (\f+ \f_t)/2, && p = t/2,\\  
&\nu = (\lam,t) \in N_2 \then &&\tau = (\f+ \f_t)/(2k), && p = t/(2k),\\
&\nu = (\lam,t) \in N_3 \then &&\tau = (\f+ \f_t)/2, && p = t/2.    
\end{align*}

\begin{proposition}
\label{propos:lami=lam}
Let $(\lam,t) \in N$, $\eps^i(\lam,t) = (\lam^i,t) \in N$. Then:
\begin{itemize}
\item[$(1)$]
$\lam^1 = \lam \iff 
\begin{cases}
\tc = 0, & \lam \in C_1, \\
\text{is impossible for} & \lam \in C_2\cup C_3,
\end{cases}
$
\item[$(2)$]
$\lam^2 = \lam \iff 
\begin{cases}
\ts = 0, & \lam \in C_1 \cup C_2, \\
\tau = 0 & \lam \in  C_3,
\end{cases}
$
\item[$(3)$]
$\lam^5 = \lam$ is impossible, 
\item[$(4)$]
$\lam^6 = \lam \iff 
\begin{cases}
\text{is impossible for} & \lam \in C_1\cup C_3,\\
\tc = 0, & \lam \in  C_2.
\end{cases}
$ 
\end{itemize}
\end{proposition}
\begin{proof}
We prove only item (1), all other items are proved similarly. 

By Propos.~\ref{propos:lami}, if $\lam \in C^i_1$, then $\lam^1 \in C^i_1$, $i = 0, 1$. Moreover,  
$$
\lam^1 = \lam 
\iff
\begin{cases}
\g_t = \g\\
-c_t = c
\end{cases}
\iff
\begin{cases}
\sn \f_t  = \sn \f\\
-\cn \f_t = \cn \f
\end{cases}
\iff
\tc = 0.
$$

If $\lam \in C_2^{\pm}$, then $\lam^1 \in C_2^{\mp}$, thus the equality $\lam^1 = \lam$ is impossible.

Similarly, if $\lam \in C_3^{i\pm}$, then $\lam^1 \in C_3^{i\mp}$, $i = 0, 1$, and the equality $\lam^1 = \lam$ is impossible.
\end{proof}

\subsection
[General description of Maxwell strata  generated by reflections]
{General description of Maxwell strata \\ generated by reflections}

We summarize our computations of the previous subsections.

\begin{theorem}
\label{th:Max_gen}
Let $\nu= (\lam,t) \in \cup_{i=1}^3 N_i$ and $q_t = (x_t,y_t,\theta_t) = \Exp(\nu)$.
\begin{itemize}
\item[$(1)$]
$\nu \in \MAX^1 
\iff
\begin{cases}
R_1(q_t) = 0, \ \tc \neq 0, & \text{for } \lam \in C_1, \\
R_1(q_t) = 0,               & \text{for } \lam \in C_2 \cup C_3.
\end{cases}
$
\item[$(2)$]
$\nu \in \MAX^2 
\iff
R_2(q_t) = 0, \ \ts \neq 0.
$
\item[$(3)$]
$\nu \in \MAX^5 
\iff
\theta_t = \pi  \text{ or } (x_t,y_t,\theta_t) = (0,0,0).
$
\item[$(4)$]
$\nu \in \MAX^6 
\iff
\begin{cases}
\theta_t = 0, & \text{for } \lam \in C_1 \cup C_3, \\
\theta_t = 0, \ \tc \neq 0  & \text{for } \lam \in C_2.
\end{cases}
$
\end{itemize}
\end{theorem}
\begin{proof}
Apply Propositions~\ref{propos:qt=qit} and~\ref{propos:lami=lam}.
\end{proof}

\subsection{Complete description of Maxwell strata}
\label{subsec:Max_compl}

We obtain bounds for roots of the equations $R_i(q_t) = 0$, $\sin \theta_t = 0$ that appear in the description of Maxwell strata given in Th.~\ref{th:Max_gen}.

We use the following representations of functions along extremal trajectories obtained by direct computation.

If $\lam \in C_1$, then
\begin{align}
& \sin \theta_t = - s_1 \cdot 2 \cc \ss \td / \Del, \label{sinthC1} \\
& \cos (\theta_t/2) = s_3 \cdot \cc / \sqrt{\Del}, \label{costh2C1} \\
& \sin (\theta_t/2) = s_4 \cdot \ss \td / \sqrt{\Del}, \label{sinth2C1} \\
& R_1(q_t) = - s_3 \cdot 2(p -\E(p)) \tc /(k \sqrt{\Del}), \label{R1C1} \\
& R_2(q_t) = - s_4 \cdot 2 f_2(p,k)  \ts /(k \sqrt{\Del}), \label{R2C1} \\
& f_2(p,k) = k^2 \cc \ss - \dd(p-\E(p)), \nonumber\\ 
& \Del = 1 - k^2 \ssp \tdp, \nonumber\\
& s_3 = \pm 1, \quad s_4 = \pm 1, \quad s_1 = -s_3 s_4. \nonumber
\end{align}

If $\lam \in C_2$, then
\begin{align}
& \sin \theta_t = - 2 k \ss \dd \tc / \Del, \label{sinthC2} \\
& \cos (\theta_t/2) = s_3 \cdot \dd / \sqrt{\Del}, \label{costh2C2} \\
& \sin (\theta_t/2) = s_4 \cdot k \ss \tc / \sqrt{\Del}, \label{sinth2C2} \\
& R_1(q_t) = s_2 s_4  \cdot 2(p -\E(p)) \td /\sqrt{\Del}, \label{R1C2} \\
& R_2(q_t) = s_2 s_4 \cdot 2 k f_1(p,k)  \ts /\sqrt{\Del}, \label{R2C2} \\
& f_1(p,k) = \cc (\E(p)-p)   - \dd\ss,                   \label{f1pk}\\ 
& s_3 = - s_4 = \pm 1. \nonumber
\end{align}

\begin{proposition}
\label{propos:thetat=0}
Let $t > 0$.
\begin{itemize}
\item[$(1)$]
If $\lam \in C_1$, then $\theta_t = 0 \iff p = 2 Kn$.
\item[$(2)$]
If  $\lam \in C_2$, then $\theta_t = 0 \iff (p = 2Kn \text{ or } \tc = 0)$.
\item[$(3)$]
If   $\lam \in C_3$, then $\theta_t = 0$ is impossible.
\end{itemize}
\end{proposition}
\begin{proof}
Apply~\eq{sinth2C1} in item (1), \eq{sinth2C2} in item (2), and pass to the limit $k \to 1 - 0$ in item (3).
\end{proof}

\begin{proposition}
\label{propos:thetat=pi}
Let $t > 0$.
\begin{itemize}
\item[$(1)$]
If $\lam \in C_1$, then $\theta_t = \pi \iff p =  K + 2 Kn$.
\item[$(2)$]
If  $\lam \in C_2$, then $\theta_t = \pi$  is impossible.  
\item[$(3)$]
If   $\lam \in C_3$, then $\theta_t = \pi$ is impossible.
\end{itemize}
\end{proposition}
\begin{proof}
Apply~\eq{costh2C1} in item (1), \eq{costh2C2} in item (2), and pass to the limit $k \to 1 - 0$ in item (3).
\end{proof}

\begin{lemma}
\label{lem:p-E>0}
For any $k \in (0,1)$ and any $p > 0$ we have $p - \E(p) > 0$.
\end{lemma}
\begin{proof}
$\ds p - \E(p) = p - \int_0^p \dn^2 t \, dt = k^2 \int_0^p \sn^2 t \, dt > 0$.
\end{proof}

\begin{proposition}
\label{propos:R1=0}
Let $t > 0$.
\begin{itemize}
\item[$(1)$]
If $\lam \in C_1$, then $R_1(q_t) = 0 \iff \tc =0$.
\item[$(2)$]
If  $\lam \in C_2$, then $R_1(q_t) = 0$ is impossible.
\item[$(3)$]
If   $\lam \in C_3$, then $R_1(q_t) = 0$ is impossible.
\end{itemize}
\end{proposition}
\begin{proof}
Apply~\eq{R1C1} and Lemma~\ref{lem:p-E>0} in item (1); \eq{R1C2} and Lemma~\ref{lem:p-E>0} in item (2); and pass to the limit $k \to 1 - 0$ in item (3).
\end{proof}

\begin{lemma}
\label{lem:f2=0}
For any $k \in (0,1)$ and $p > 0$ we have $f_2(p,k) > 0$.
\end{lemma}
\begin{proof}
The function $f_2(p)$ has the same zeros as the function $g_2(p) = f_2(p) /\dd$. But $g_2(p) > 0$ for $p > 0$ since $g_2(0) = 0$ and $g_2'(p) = k^2 \ccp/ \ddp \geq 0$.
\end{proof}

\begin{lemma}
\label{lem:f1=0}
For any $k \in [0, 1)$, the function $f_1(p)$ has a countable number of roots
\begin{align}
&p = p_1^n(k), \qquad n \in \Z, \nonumber \\
&p_1^0 = 0, \qquad p_1^{-n}(k) = - p_1^n(k). \label{p1-n}
\end{align}
The positive roots admit the bound
\begin{align}
&p_1^n(k) \in (-K + 2 K n, 2 Kn), \qquad n \in \N, \quad k \in (0,1), \label{p1nkin}\\
&p_1^n(0) = 2 \pi n, \qquad n \in \N. \label{p1n0}
\end{align}
All the functions $k \mapsto p_1^n(k)$, $n \in \Z$, are smooth at the segment $k \in [0, 1)$.
\end{lemma}
\begin{proof}
The function $f_1(p)$ has the same roots as the function $g_1(p) = f_1(p)/\cc$. We have
\be{f1/c'}
g_1'(p) = - \ddp/\ccp,
\ee
so the function $g_1(p)$ decreases at the intervals $p \in (-K + 2 Kn, \ K + 2 Kn)$, $n \in \Z$. In view of the limits
$$
g_1(p) \to \pm \infty \text{ as } p \to K + 2 Kn \pm 0,
$$
the function $g_1(p)$ has a unique root $p = p_1^n$ at each interval $p \in (-K + 2 Kn, K + 2 Kn)$, $n \in \Z$.

For $p = 2 K n$, $n \in \N$, we have $g_1(p) = \Eo - p < 0$, thus the bound~\eq{p1nkin} follows.


Further, equality~\eq{p1n0} follows since $f_1(p,0) = - \sin p$.

Equalities~\eq{p1-n} follow since the function $f_1(p)$ is odd.

By implicit function theorem, the roots $p_1^n(k)$ of the equation $g_1(p) = 0$ are smooth in $k$ since $g_1'(p) < 0$ when $\cc \neq 0$, see~\eq{f1/c'}.
\end{proof}

\begin{corollary}
\label{cor:p11}
\begin{itemize}
\item[$(1)$]
The first positive root of the function $f_1(p)$ admits the bound 
\be{p11kin}
p_1^1(k) \in (K(k), 2 K(k)), \qquad k \in (0, 1).
\ee
\item[$(2)$]
If $p \in (0, p_1^1)$, then $f_1(p) < 0$.
\item[$(3)$]
$\ds \lim_{k \to + 0} p_1^1(k) = \pi$, $\ds \lim_{k \to 1- 0} p_1^1(k) = + \infty$.
\end{itemize}
\end{corollary}

Plots of the functions $K(k)$, $p_1^1(k)$, $2 K(k)$  are given at Fig.~\ref{fig:p11(k)}.

\onefiglabel{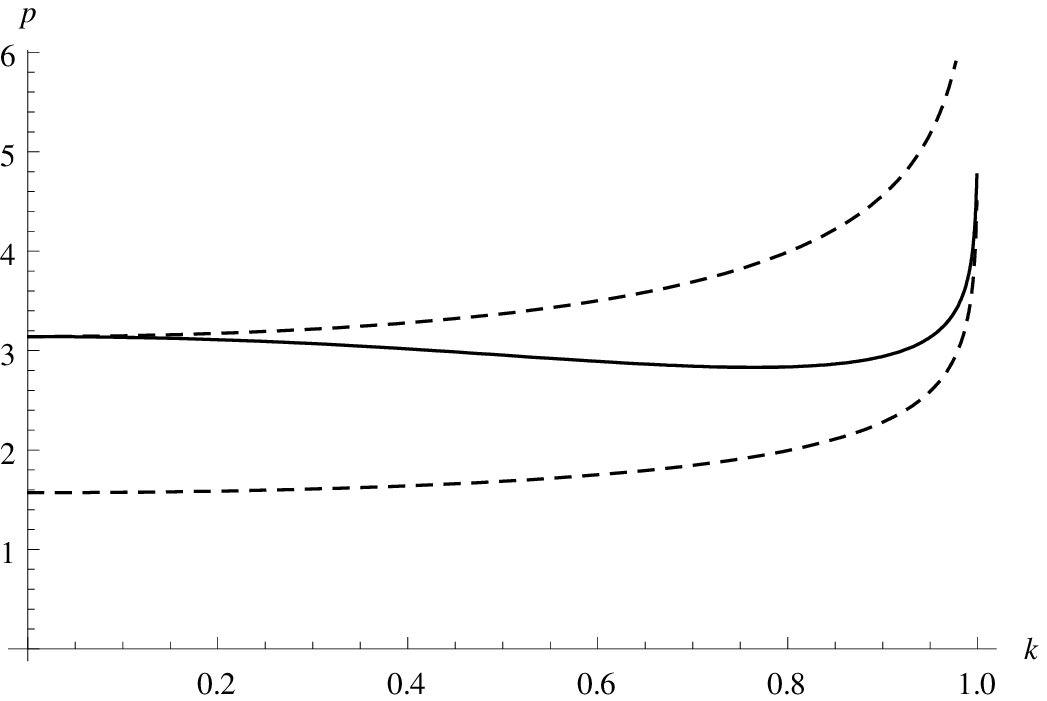}{Plots of the functions $K(k) \leq p_1^1(k) \leq 2 K(k) $}{fig:p11(k)}

\begin{proposition}
\label{propos:R2=0}
Let $t > 0$.
\begin{itemize}
\item[$(1)$]
If $\lam \in C_1$, then $R_2(q_t) = 0 \iff \ts =0$.
\item[$(2)$]
If  $\lam \in C_2$, then $R_2(q_t) = 0 \iff (p = p^n_1(k) \text{ or } \ts =0)$.
\item[$(3)$]
If   $\lam \in C_3$, then $R_2(q_t) = 0 \iff \tau =0$.
\end{itemize}
\end{proposition}
\begin{proof}
Apply~\eq{R2C1}  and Lemma~\ref{lem:f2=0} in item (1), \eq{R2C2} and Lemma~\ref{lem:f1=0} in item (2), and pass to the limit $k \to 1 - 0$ in item (3).
\end{proof}

\begin{lemma}
\label{lem:000}
If $\nu \in N_1 \cup N_2 \cup N_3$, then $(x_t,y_t,\theta_t) \neq (0,0,0)$.
\end{lemma}
\begin{proof}
The equality $(x_t,y_t,\theta_t) = (0,0,0)$ is equivalent to  $(R_1(q_t),R_2(q_t),\theta_t) = (0,0,0)$.

If $\nu \in N_1$, then the equalities $R_1(q_t) = 0$, $R_2(q_t) = 0$ are equivalent to $\tc = 0$, $\ts = 0$ (Propos.~\ref{propos:R1=0}, \ref{propos:R2=0}), which are incompatible.

If $\nu \in N_2 \cup N_3$, then the equality  $R_1(q_t) = 0$ is impossible (Propos.~\ref{propos:R1=0}).
\end{proof}


On the basis of results of the previous subsections we derive the following characterization of the Maxwell strata.

\begin{theorem}
\label{th:Max_compl}
\begin{itemize}
\item[$(1)$]
$\MAX^1 \cap N_1  =  \MAX^1 \cap N_2 =   \MAX^1 \cap N_3 = \emptyset.$ 
\item[$(2)$]
$\MAX^2 \cap N_1  =  \MAX^2 \cap N_3  = \emptyset$,

$\MAX^2 \cap N_2 = \{ \nu \in N_2 \mid p = p_1^n(k), \ts \neq 0 \}$.
\item[$(3)$]
$\MAX^5 \cap N_1  =  \{ \nu \in N_1 \mid p = K + 2 Kn \}$,

$\MAX^5 \cap N_2 = \MAX^5 \cap N_3  = \emptyset $.
\item[$(4)$]
$\MAX^6 \cap N_1  =  \{ \nu \in N_1 \mid p = 2 Kn \}$,

$\MAX^6 \cap N_2  =  \{ \nu \in N_2 \mid p = 2 Kn, \ \tc \neq 0 \}$,

$\MAX^6 \cap N_3  = \emptyset $.
\end{itemize}
\end{theorem}
\begin{proof}
Apply Th.~\ref{th:Max_gen}, Propositions~\ref{propos:thetat=0}--%
\ref{propos:R2=0}, and Lemma~\ref{lem:000}.
\end{proof}

\subsection{Upper bound on cut time}
The cut time for an extremal trajectory $q_s$ is defined as follows:
$$
\tcut = \sup \{ t_1 > 0 \mid q_s \text{ is optimal for } s \in [0,t_1]\}.
$$
For normal extremal trajectories $q_s = \Exp(\lam,s)$, the cut time is a function of the initial covector:
$$
\map{\tcut}{C}{[0, + \infty]}.
$$
Denote the first Maxwell time as
$$
\tmax(\lam) = \inf \{ t > 0 \mid (\lam,t) \in \MAX \}.
$$
A normal extremal trajectory cannot be optimal after a Maxwell point, thus
$$
\tcut(\lam) \leq \tmax(\lam) \qquad \forall \ \lam \in C.
$$
On the basis of this inequality and results of Subsec.~\ref{subsec:Max_compl}, we derive an effective upper bound on cut time in the sub-Riemannian problem on $\SE(2)$. To this end
define the following function $\map{\tt}{C}{(0, + \infty]}$:
\begin{align}
&\lam \in C_1 \then \tt(\lam) = 2 K(k), \label{ttC1} \\
&\lam \in C_2 \then \tt(\lam) = 2 k p_1^1(k), \label{ttC2}\\
&\lam \in C_3 \then \tt(\lam) = +\infty, \label{ttC3}\\
&\lam \in C_4 \then \tt(\lam) = \pi, \label{ttC4}\\
&\lam \in C_5 \then \tt(\lam) = +\infty. \label{ttC5}
\end{align}

\begin{theorem}
\label{th:tcut_bound}
Let $\lam \in C$. We have
\be{tcuttt}
\tcut (\lam) \leq \tt(\lam)
\ee
in the following cases:
\begin{itemize}
\item[$(1)$]
$\lam \in C \setminus C_2$, 
\item[$(2)$]
$\lam \in  C_2$ and $\ts \neq 0$.
\end{itemize}
\end{theorem}
\begin{proof}
If $\lam \in C_1$, then $(\lam, 4 K(k)) = (\lam, \tt(\lam))\in \MAX^6$ by item (4) of Th.~\ref{th:Max_compl}, thus
\be{tcutlamless}
\tcut(\lam) \leq \tmax(\lam) \leq  \tt(\lam).
\ee

If $(\lam,t) \in N_2$ and  $p= p_1^1(k)$, $\ts \neq 0$,  then $(\lam,t) \in \MAX^2$ by item (2) of  Th.~\ref{th:Max_compl}, and the chain~\eq{tcutlamless} follows.

If $\lam \in C_4$, then the trajectories $\Exp(\lam,t) = (0,0,-s_1 t)$ and $\Exp(\lam^4,t) = (0,0,s_1 t)$ intersect one another at the instant $t = \pi$, thus $(\lam,\pi) = (\lam,  \tt(\lam)) \in \MAX^4$, and the chain~\eq{tcutlamless} follows as well.
\end{proof}

\subsection{Limit points of Maxwell set}
Here we fill the gap appearing in item~(2) of Th.~\ref{th:tcut_bound} via the theory of conjugate points.

A normal extremal trajectory (geodesic)
$q_t$  is called \ddef{strictly normal} if it is a projection of a normal extremal~$\lam_t$,  but is not a projection of an abnormal extremal. In the sub-Riemannian problem on $\SE(2)$ all geodesics   are strictly normal.  

A point 
$q_t$ of a strictly normal geodesic 
$q_s = \Exp(\lam, s)$, $s \in [0, t]$, is called \ddef{conjugate} to the point 
$q_0$ along the geodesic
$q_s$ if  $\nu = (\lam, t)$ is a critical point of the exponential mapping.

It is known that a strictly normal geodesic cannot be optimal after a conjugate point~\cite{notes}. At the first conjugate point a geodesic loses its local optimality. Below we find conjugate points on geodesics with 
$\lam \in C_2$ not containing Maxwell points. These conjugate points are limits of pairs of the corresponding Maxwell points, the corresponding theory was developed in~\cite{max3}.

\begin{proposition}[Propos. 5.1~\cite{max3}]
\label{propos:prop5.1max3}
Let
$\nu_n, \ \nu'_n \in N$, $\nu_n \neq \nu'_n$, $\Exp(\nu_n) = \Exp(\nu'_n)$, $n \in \N$. If the both  sequences
$\{\nu_n\}$, $\{\nu'_n\}$ converge to a point 
$\bnu = (\lam, t)$, and the geodesic 
$q_s = \Exp(\lam, s)$ is strictly normal, then its endpoint 
$q_t= \Exp(\bnu) $ is a conjugate point.
\end{proposition}

It is convenient to introduce the following set, which we call the \ddef{double closure of Maxwell set}:
\begin{multline*}
\CMAX = \left\{
\bnu \in N \mid 
\exists \ \{\nu_n = (\lam_n, t_n)\}, \ 
\{\nu'_n = (\lam'_n, t_n)\} \subset N \ : 
\vphantom{\Exp(\nu'_n) \lim\limits_{n \to \infty}}\right.
\\ 
\left.
\nu_n \neq \nu'_n, 
\Exp(\nu_n) = \Exp(\nu'_n), \ n \in \N, \ 
\lim\limits_{n \to \infty} \nu_n = \lim\limits_{n \to \infty} \nu'_n = \bnu 
\right\}.
\end{multline*}
It is obvious that 
$\nu_n \in \MAX$, thus 
$\CMAX \subset \cl(\MAX)$.

Proposition~\ref{propos:prop5.1max3} claims that if 
$\nu = (\lam,t) \in \CMAX$ and the geodesic 
$q_s = \Exp(\lam, s)$ is strictly normal, then its endpoint 
$q_t$ is a conjugate point.

\begin{proposition}
\label{propos:CMAXN2}
Let $\nu = (\lam,t) \in N_2$ be such that $p = p_1^1(k)$, $\ts = 0$. Then the point $q_t = \Exp(\nu)$ is conjugate, thus $t \geq \tcut(\lam)$.
\end{proposition}
\begin{proof}
Consider the points $\nu_n^{\pm} = (p_1^1(k), \tau \pm 1/n, k) \in N_2$. Then $\nu_n^+ \neq \nu_n^-$ and $\lim_{n \to \infty} \nu^{\pm} = \nu$. Formulas~\eq{sinthC2}--\eq{R2C2} imply that $\Exp(\nu_n^-) = \Exp(\nu_n^+)$. Thus $\nu \in \CMAX$, and the statement follows from Propos.~\ref{propos:prop5.1max3}.
\end{proof}

\subsection{The final bound of the cut time}

\begin{theorem}
\label{th:tcut_bound_fin}
There holds the bound 
\be{tcut_bound_fin}
\tcut(\lam) \leq \tt(\lam) \qquad \forall \lam \in C.
\ee
\end{theorem}
\begin{proof}
Apply Th.~\ref{th:tcut_bound} and Propos.~\ref{propos:CMAXN2}.
\end{proof}

The function $\tt(\lam)$ deserves to be studied in some detail. One can see from its definition~\eq{ttC1}--\eq{ttC5} that  the function $\tt$ depends only on the elliptic coordinate $k$, i.e., only on the energy $E$~\eq{E} of pendulum~\eq{ham_vert}, but not on its phase $\f$. Thus we have a function
$$
\mapto{\tt}{E}{\tt(E)}, \qquad \map{\tt}{[-1,+\infty)}{(0, + \infty]}.
$$

\begin{proposition}
\label{propos:ttE}
\begin{itemize}
\item[$(1)$]
The function $\tt(E)$ is smooth for $E \in [-1, 1) \cup (1, + \infty)$.
\item[$(2)$]
$\lim\limits_{E \to -1+0} \tt(E) = \pi$;    $\lim\limits_{E \to 1} \tt(E) = + \infty$;   $\tt \sim 2 \sqrt 2 \pi /\sqrt{E+1} \to 0 $ as $E \to + \infty$.
\end{itemize}
\end{proposition}
\begin{proof}
(1) follows from smoothness of the functions $K(k)$ and $p_1^1(k)$ for $k \in [0,1)$. 

(2) follows from the limits
$\lim\limits_{k \to +0} K(k) = \pi/2$, $\lim\limits_{k \to 1-0} K(k) = \lim\limits_{k \to 1-0} p_1^1(k) =+ \infty$, $\lim\limits_{k \to 0} p_1^1(k) = 2 \pi$.
\end{proof}

A plot of the function $\tt(E)$ is given at Fig.~\ref{fig:tt}.

\onefiglabel{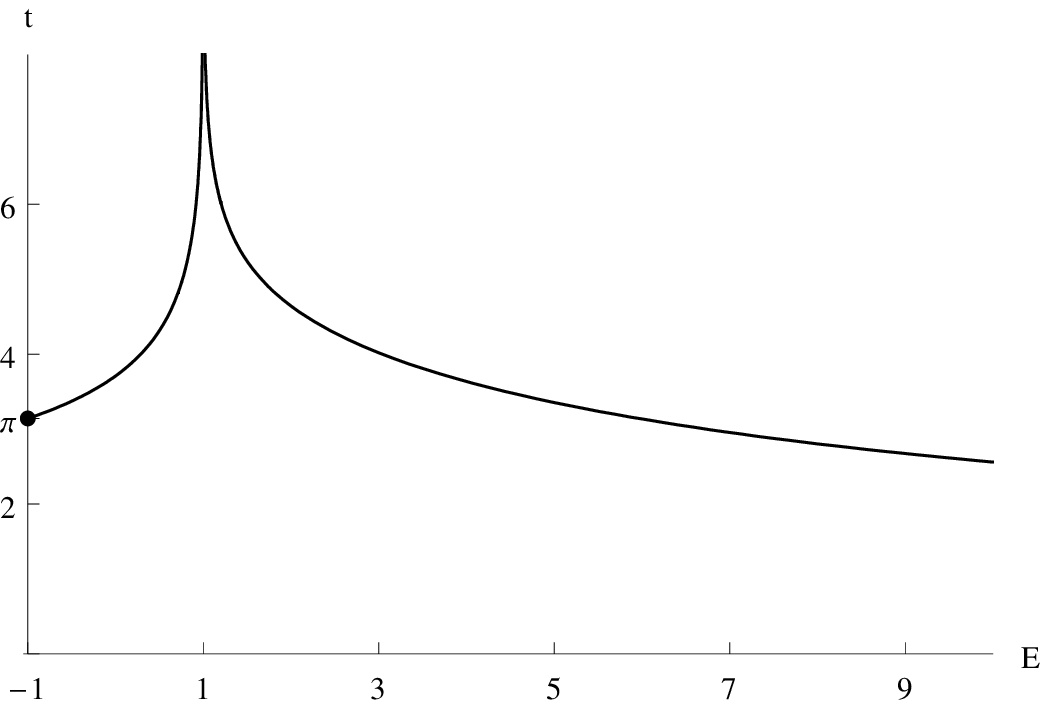}{Plot of the function  $E \mapsto \tt(E)$}{fig:tt}

In our forthcoming work~\cite{cut_sre} we show that the inequality ~\eq{tcut_bound_fin} is in fact an equality, i.e., $\tcut(\lam) = \tt(\lam)$ for $\lam \in C$.

\newpage
\addcontentsline{toc}{section}{\listfigurename}
\listoffigures

\end{document}